\documentclass{article}
\usepackage{inputenc, amsmath}
\usepackage{amsthm}
\usepackage{xcolor}
\usepackage{amssymb, mathrsfs}
\usepackage{geometry}
\usepackage{enumitem}  
\usepackage{lineno}
\usepackage[affil-it]{authblk}

\newtheorem{theorem}{Theorem}[section]
\newtheorem{defn}[theorem]{Definition}

\newtheorem{fact}[theorem]{Fact}
\newtheorem{cor}[theorem]{Corollary}
\newtheorem{lemma}[theorem]{Lemma}

\newtheorem{claim}[theorem]{Claim}
\newtheorem{example}[theorem]{Example}
\newtheorem{question}[theorem]{Question}

\providecommand{\keywords}[1]{\textbf{\textit{Key words.}} #1}

\renewcommand{\epsilon}{{\varepsilon}}
\newcommand{\Z}{{\mathbb{Z}}}
\newcommand{\E}{{\mathbb{E}}}

\title{Tight minimum colored degree condition for rainbow connectivity}
\author{Andrzej Czygrinow\thanks{aczygri@asu.edu} and Xiaofan Yuan\thanks{xiaofan.yuan@asu.edu}}
\affil{School of Mathematical and Statistical Sciences, Arizona State University}

\date{}

\begin{document}

\maketitle
\begin{abstract}
Let $G = (V,E)$ be a graph on $n$ vertices, and let $c : E \to P$, where $P$ is a set of colors.  
Let $\delta^c(G) = \min_{v \in V} \{ d^{c}(v) \}$ where $d^c(v)$ is the number of colors on edges incident to a vertex $v$ of $G$.  In 2011, Fujita and Magnant showed that if $G$ is a graph on $n$ vertices that satisfies $\delta^c(G)\geq n/2$, then for every two vertices $u, v$ there is a properly-colored $u,v$-path in $G$.
In this paper, we show that the same bound for $\delta^c(G)$ implies that any two vertices are connected by a rainbow path.
\end{abstract}


\section{Introduction} 

An \textit{edge-colored graph} is a pair $(G, c)$, where $G = (V, E)$ is a graph and $c: E \to P$ is a function mapping edges to a color-set $P$. A subgraph $H \subseteq G$ is a \textit{rainbow subgraph} if all edges of $H$ receive distinct colors in $c$. 
In an edge-colored graph $(G, c)$, let $d^c(v)$ be the number of colors on the edges incident to $v$, and let $\delta^c(G):=\min_{v\in V(G)} d^c(v)$. 
For $k\in \mathbb{Z}^+$, an edge-colored graph $(G,c)$ is called \textit{rainbow $k$-connected} if for any two distinct vertices $u,v$ in $G$ there are $k$ internally disjoint $u,v$-paths in $G$ such that the subgraph induced by the paths is rainbow. If $(G,c)$ is $1$-rainbow connected, then we call it \textit{rainbow-connected}.
In this paper, we will give a tight sufficient condition for $\delta^c(G)$ that implies $G$ is rainbow-connected.

Recently, there has been a lot of interest in studying rainbow and properly-colored structures in edge-colored graphs.
The meta-question that can be considered in this context is the following. For a property $T$ and $n\in \mathbb Z^+$, determine the least $\delta(n,T)$ such that if an edge-colored graph $G$ of order $n$ satisfies $\delta^c(G)\geq \delta(n,T)$, then $G$ has property $T$. For example, H. Li \cite{H-Li} proved that if $T$ is the existence of a rainbow triangle, then $\delta(n,T)=({n+1})/{2}$ and coincides with the bound for the minimum degree in an uncolored graph that results in a triangle. However, the bounds for uncolored structures and their rainbow analogs can be far apart. For example, as proved in \cite{CNOM},  one needs {$\delta(n,T)\geq ({n+5})/{3}$} when $T$ is the existence of a rainbow cycle of length four.

The situation is much better understood if an underlying edge-coloring of a graph $G$ is proper. Keevash, Mubayi, Sudakov, and Verstra\"{e}te \cite{KMSV} introduced rainbow Tur\'{a}n numbers and showed that when confined to proper edge-colorings, the rainbow Tur\'{a}n number of a fixed non-bipartite graph $H$ is approximately the same as the regular Tur\'{a}n number of $H$.
We refer to the survey of Fujita et al. \cite{magnant-survey} for a comprehensive overview of rainbow-related problems.

In this paper, we are interested in the case when $T$ is the property of being rainbow-connected. 
Problems on properly-colored connectivity have also been considered.  For a colored graph $(G,c)$ we say that $G$ is \textit{properly connected} if for any two distinct vertices $u,v$ there is a properly-colored $u,v$-path in $G$.
In \cite{FM}, Fujita and Magnant gave a tight bound for $\delta^c(G)$ that implies proper connectivity. Specifically, they proved the following fact.
\begin{theorem}(\cite{FM})\label{thm-FM}
Let $(G,c)$ be an edge-colored graph of order $n\geq 3$. If $\delta^c(G) \geq n/2$, then $G$ is properly connected.
\end{theorem}
The following example shows that the bound in Theorem \ref{thm-FM} is the best possible.
\begin{example}(\cite{FM})
Let $n$ be odd. Consider the tournament $T$ obtained from a regular tournament on a set $V$ with $|V|=n-2$ by adding two new vertices $x,y$ and all arcs of the form $zw$ where $z\in\{x,y\}$ and $w\in V$. Now, let the edge-colored graph $G$ be obtained by coloring arc $uv$ with color $c_v$, where $c_v\neq c_u$ for different $u$ and $v$. We have $\delta^c(G)=\frac{n-1}{2}$ and there is no properly-colored $x,y$-path in $G$. 
\end{example}
Cheng, Kano, and Wang \cite{CKW} showed a stronger version of Theorem~\ref{thm-FM} and proved that 
if $G$ is a graph of order $n$ such that $\delta^c(G)\geq n/2$, then $G$ contains a properly colored spanning tree.
{However, it is not possible to give a non-trivial bound for $\delta^c(G)$ that would imply that $G$ has a rainbow spanning tree. Indeed, when $n$ is even, a disjoint union of $n-2$ perfect matchings, where each matching receives a unique color, does not have a rainbow spanning tree. }
On the other hand, if an edge-coloring of a graph $G$ of order $n$ with $\delta^c(G)\geq n/2$ satisfies additional assumptions, then a rainbow spanning tree exists. For example, Cheng et al. \cite{CKW} proved that if $G$ is an edge-colored graph of order $n$ such that $\delta^c(G)\geq n/2$ and for every color class $C$ the number of vertices in $G[C]$ is at most  $n/2+1$, then $G$ has a rainbow spanning tree.  This immediately gives the aforementioned result on properly colored spanning trees, and also implies that if $\delta^c(G)\geq n/2$ and every color class is an induced matching, then $G$ has a rainbow spanning tree (initially proved in \cite{CSTW}).  
The proof in \cite{CKW} relies on the fact proved by Akbari and Alipour \cite{AA} and independently by Suzuki \cite{Suz} who showed that an edge-colored connected graph $G$ of order $n$ has a rainbow spanning tree if and only if for every $1\leq r\leq n-2$ and every set of $r$ colors $C$ removing all edges of colors from $C$ results in a graph with at most $r+1$ components.

Surprisingly, in the case of rainbow connectivity, no additional assumptions are needed. In fact, we can
prove that if $\delta^c(G)\geq n/2$ then for every two distinct vertices $u,v$ in $G$, there is a rainbow $u,v$-path of length $O(1)$ as following.  

\begin{theorem}\label{main}
There exists $n_0\in \mathbb{Z}^+$ such that, for every $n\geq n_0$,
if $(G, c)$ is an edge-colored graph of order $n$ and $\delta^c(G) \geq n/2$, then any two distinct vertices are connected by a rainbow path of length at most nine.
\end{theorem}
To prove Theorem \ref{main}, we use the stability method and argue that either an edge-colored graph has a specific structure or the bound of $(1/2-\epsilon)n$ for the colored degree is already sufficient to guarantee rainbow paths between any two vertices. To prove the latter, given an edge-colored graph $(G,c)$, we consider two auxiliary graphs on $V(G)$, an oriented graph and a graph, both arising from the coloring $c$.  We regularize the two auxiliary graphs simultaneously and analyze the reduced graphs. {Either both of the auxiliary graphs are inherited in the reduced graphs, or one is empty, and we end up in one of the extremal cases.}

For the property of being rainbow $k$-connected, we show that given $k$, there is a $C_k$ such that for $n$ large enough, if $G$ is a graph of order $n$ that satisfies $\delta^c(G)\geq n/2+C_k$, then for any two vertices there are $k$ rainbow paths between them.  
Because the paths obtained in Theorem \ref{main} have length at most {nine}, it is possible to simply iterate  Theorem \ref{main} and prove the following.
\begin{cor}\label{cor-kconn}
For every $k\in \mathbb{Z}^+$ there exists $n_0\in \mathbb{Z}^+$ such that for every $n\geq n_0$, if $(G,c)$ is an edge-colored graph on $n\geq n_0$ vertices and $\delta^c(G)\geq {n}/{2}+17(k-1)$, then $(G,c)$ is rainbow $k$-connected.  
\end{cor}
\begin{proof} 
Fix $u\neq v$. We will apply Theorem \ref{main} $k$ times to $u,v$, modifying $G$ after each iteration. For the $i^{th}$ iteration, suppose $\delta^c(G)\geq {n}/{2}+ 17(k-i)$ and let $P_i$ denote a rainbow $u,v$-path of length at most nine. Remove all internal vertices in $P_i$ from $G$ and all 
edges of $G$ in the colors used by the edges of $P_i$.
Then for every vertex $w$ there are at most $(8+9)$ colors removed on edges incident to $w$ in each iteration.
Consequently, after the modification, we have $\delta^c(G)\geq {n}/{2}+ 17(k-i-1).$ 
\end{proof}
In the next section, we introduce more notation and terminology. Section \ref{reg-section} contains a discussion of the regularity-related concepts, the non-extremal case is addressed in Section \ref{non-section}, and the extremal configurations are analyzed in Section \ref{ext-section}.

\section{Preliminaries}\label{prem-section}

{Let $G=(V, E)$ be a graph and $D=(V, F)$ be a di-graph. For a vertex $v\in V$, let $N_G(v)$ denote the set of neighbors of $v$ in $G$, and let $N^+_D(v)$ ($N^-_D(v)$) be the set of out-neighbors (in-neighbors) of $v$ in $D$. 
Let $d_G(v) := |N_G(v)|$, $d^+_D(v) := |N^+_D(v)|$, and $d^-_D(v) := |N^-_D(v)|$.

Let $\delta(G) := \min_{v \in V} |N_G(v)|$ and $\delta^+(D):= \min_{v \in V} |N_D^+(v)|$. If there is an arc from $u$ to $v$ in $D$, we denote it by $uv$, and if there is an edge between $u$ and $v$ in $G$, we denote it by $\{u,v\}$. If $X, Y$ are subsets of $V(D)$, then let $E_D(X,Y)$ denote the set of arcs $xy$ in $D$ such that $x\in X$ and $y\in Y$. Let $D[X, Y]$ denote the subdigraph of $D$ defined by $D[X, Y]:=(X\cup Y, E_D(X,Y))$. Similarly, for a graph $G$, $E_G(X,Y)$ denotes the set of edges with one endpoint in $X$ and the other in $Y$, and $G[X, Y]:=(X\cup Y, E_G(X,Y))$.

Let $P$ be a set of colors and let $(G,c)$ be an edge-colored graph on $n$ vertices where $c:E(G)\rightarrow P$.}  

For a vertex $v\in V(G)$, let   $d^c(v)$ be the number of colors on the edges incident to $v$, and let $\delta^c(G):=\min_{v\in V(G)} d^c(v)$. We will always assume that $(G,c)$ satisfies  $\delta^c(G)\geq \frac{n}{2}$ and subject to this condition, $E(G)$ is minimal. As a result, $G$ does not contain a monochromatic cycle of any length and does not contain a monochromatic path of length three. 
Indeed, if $uvwz$ is a monochromatic path of length three, then deleting $uw$ from $G$ does not affect the minimum colored degree of $G$. Similarly, if $uvz$ is a monochromatic triangle, we can remove $uv$. Consequently, for every color $\alpha \in c(E(G))$, the graph $F_{\alpha}:= (V(G), c^{-1}(\alpha))$ is a star forest. 
We will consider a digraph associated with $(G,c)$ that will be critical to our analysis of the rainbow connectivity. 
\begin{defn}\label{defnD}
Let $(G,c)$ be a graph on $n$ vertices. Let $D_G:=(V(G), E)$ be the digraph obtained as follows. For every $\alpha \in c(E(G))$ and every $v\in V(G)$ that satisfies $1\leq d_{F_{\alpha}}(v)\leq \sqrt{n}$, choose one vertex $w\in N_{F_{\alpha}}(v)$ and add $vw$ to E.  
\end{defn}
  It is not difficult to see the following fact.
\begin{fact}\label{fact-outdeg}
Let $(G,c)$ be an edge-colored graph on $n$ vertices such that $\delta^c(G)\geq \frac{n}{2}$, and subject to this $E(G)$ is minimal. Then $\delta^+(D_G)> n/2 -\sqrt{n}.$
\end{fact}
\begin{proof}
For every $v\in V(G)$, $d_{F_{\alpha}}(v)\geq 1$ for at least $d^c(v)$ colors $\alpha.$ In addition, $d_{F_{\alpha}}(v)>\sqrt{n}$ for less than $\sqrt{n}$ colors $\alpha.$ Therefore, $d_{D_G}^+(v)> n/2 -\sqrt{n}.$
\end{proof}

In addition to $D_G$, we shall consider the following graph associated with $(G,c)$. 
\begin{defn} \label{defG*}Let $G^*= (V(G), E)$ be the graph such that $\{u,v\}\in E$ if both $uv$ and $vu$ are in $E(D_G)$.
\end{defn}
The idea behind our approach is to use the regularity lemma applied to both $D_G$ and $G^*$ to either obtain rainbow paths between any two vertices or to argue that $D_G$ and $G^*$ are specific extremal configurations. There will be two extremal cases. In the first one, the vertex set of {$G$ can be partitioned into two sets $V_1, V_2$, each of size approximately $n/2$ so that the number of arcs in $D_G$ between $V_1, V_2$ is very small.} In the second extremal case, $G^*$ is approximately empty, and $D_G$ is approximately a {regular} tournament. The former extremal case is relatively easy to address, yet the latter requires a more involved argument.
\begin{defn}
    Let $\beta\in (0,1)$. We say that $(G,c)$ is $\beta$-extremal of type 1 if there is a partition $\{V_1,V_2\}$ of $V(G)$ such that for $i=1,2$, $|V_i|\geq (1/2-\beta)n$, and the number of arcs {in $D_G$} with one endpoint in $V_1$ and the other in $V_2$ is at most $\beta n^2.$ 
\end{defn}
\begin{defn}Let $\beta\in (0,1)$. We say that $(G,c)$ is $\beta$-extremal of type 2 if  $|E(G^*)|\leq \beta n^2.$
\end{defn}
We say that $(G,c)$ is not $\beta$-extremal if it is not extremal of type 1 or extremal of type 2.

\section{Regularity lemmas}\label{reg-section}
 We will discuss concepts related to the regularity lemma for graphs and digraphs. We will only consider {\it simple digraphs}, that is, digraphs such that for every pair of vertices $(u,v)$ there is at most one arc from $u$ to $v$. A digraph $D$ is called an {\it oriented graph}, if for every two vertices $u,v$ in $D$, at most one of the arcs $uv$, $vu$ is in $D$.

Let $G=(V,E)$ be a graph and let $V_1, V_2$ be two disjoint non-empty subsets of $V$. Then the {\it density} of the pair $(V_1,V_2)$ is defined as 
\[ d_G(V_1,V_2):=\frac{|E_G(V_1, V_2)|}{|V_1||V_2|}.\] 
Let $0\leq \epsilon, d\leq 1$. 
We say that $G[V_1,V_2]$ is \textit{$\epsilon$-regular of density $d$} if for every $V_1'\subseteq V_1$ and every $V_2'\subseteq V_2$ with $|V_i'|\geq \epsilon |V_i|$ for $i=1,2$,
\[|d_{G}(V_1', V_2')|-d|\leq \epsilon.\]

We similarly define $\epsilon$-regularity in digraphs. 
Let $D=(V,F)$ be a digraph and let $V_1, V_2$ be two disjoint non-empty subsets of $V$. Then the {\it density} of the ordered pair $(V_1,V_2)$ is defined by 
\[d_D(V_1,V_2):=\frac{|E_D(V_1, V_2)|}{|V_1||V_2|}.\]
We say that $D[V_1,V_2]$ 
is \textit{$\epsilon$-regular of density $d$} if for every $V_1'\subseteq V_1$ and $V_2'\subseteq V_2$ with $|V_i'|\geq \epsilon |V_i|$ for $i=1,2$, 
\[|d_{D}(V_1', V_2')|-d|\leq \epsilon.\]

We say that an oriented graph $D=(V,F)$ and a graph $G=(V,E)$ are \textit{edge-disjoint} if the underlying undirected graph of $D$ and $G$ are edge-disjoint.
The main tool used to address the non-extremal case will be the regularity lemma of Szemer\'{e}di \cite{szem}. We will consider an oriented graph $D$ and a graph $G$ on the same set of vertices such that $D$ and $G$ are edge-disjoint, and regularize both $D$ and $G$. {The following lemma is a modification of the Diregularity Lemma of Alon-Shapira from \cite{alon-shapira}. It is essentially the degree form of the Diregularity Lemma (Lemma 3.1 in \cite{KKO}) with an additional refinement of a di-regular partition of $D$ that yields the regularity of $G$, along the same lines as in Theorem 3.1 from \cite{CMN}. }
\begin{lemma}[Degree form of the regularity lemma]
\label{deg-form}
    Let $\epsilon \in (0,1)$ and let $m\in \Z^+$. Then there exist integers $M$ and $n_0$ such that the following holds. If $D=(V,F)$ is an oriented graph and $G=(V,E)$ is a graph on $|V|\geq n_0$ vertices such that $D$ and $G$ are edge-disjoint, and $d\in (0,1)$, then there is a partition $V_0,V_1, \dots, V_t$ of $V$, a spanning subdigraph $D'$ of $D$, and a spanning subgraph $G'$ of $G$ such that
    \begin{enumerate}[label=(\roman*)]
    \setlength\itemsep{0em}
    \item $m\leq t\leq M$,
    \item $|V_0|\leq \epsilon |V|$,
    \item $|V_1|=\cdots =|V_t|$,
    \item $d_{G'}(x)\geq d_G(x) -{(d+\epsilon)|V|}$ and  $d^+_{D'}(x)\geq d_D^+(x)-{(d+\epsilon)}|V|$ for every $x\in V$,
    \item all but at most $\epsilon t^2$ pairs $(V_i, V_j)$ are $\epsilon$-regular in  both $D$ and in $G$,
    \item $G'[V_i]=D'[V_i]=\emptyset$ for every $i\in [t]$,
    \item for every $1\leq i\neq j\leq t$, $G'[V_i, V_j]$ is $\epsilon$-regular of density zero or at least $d$,
    \item for every $1\leq i\neq j\leq t$, $D'[V_i, V_j]$ is $\epsilon$-regular of density zero or at least $d.$
    \end{enumerate}
\end{lemma}
{Given $\epsilon, d\in [0,1]$ and $m\in \Z^+$  and a digraph $D_G$, we consider two reduced graphs that are obtained from Lemma \ref{deg-form}}. The first reduced graph arises from $G:=G^*$ (Definition \ref{defG*}), and the second is a reduced digraph associated with the digraph $D$ obtained from $D_G$ by removing arcs $uv$ such that $\{u,v\}\in E(G^*)$. Specifically,  the {\it reduced digraph of $D$}, $R'(D):=R'(D,m, \epsilon, d)$, is the digraph with $V(R'(D))=[t]$ obtained by adding $ij$ to $E(R'(D))$ when $d_{D'}(V_i,V_j)>0$, and {\it the reduced graph of $G$}, $R'(G):=R'(G,m, \epsilon, d)$, is the graph with $V(R'(G))=[t]$ obtained by adding {$\{i,j\}$} to $E(R'(G))$ if $d_{G'}(V_i,V_j)>0$. Note that $R'(D)$ can contain both $ij$ and $ji$. 

Let $N_{R'(G)}(i)$ denote the set of neighbors of $i$ in $R'(G)$, and $deg_{R'(G)}(i):=|N_{R'(G)}(i)|$. 
Let $s:= |V_1|=\cdots =|V_t|$.
The next lemma is a small modification of Lemma 3.2 from {\cite{KKO}} that was suggested by T. Molla \cite {TM}, and allows us to find a spanning oriented subgraph of $R'(D)$ that has useful properties.

\begin{lemma}\label{cluster-graph-lem}For every $\epsilon\in (0,1)$ there exist $m$ and $n_0$ such that the following holds. Let $d\in (0,1)$, $D=(V,F)$ be an oriented graph, and let $G=(V,E)$ be a graph such that  $|V|\geq n_0$ and $D$ and $G$ are edge-disjoint. Then $R'(D,m, \epsilon, d)$ defined above contains a spanning oriented subgraph $R := R(D,m, \epsilon, d; G)$ such that $R$ and $R'(G)$ are edge-disjoint, and for every $x\in V(R)$,  
$$deg^+_{R}(x)+deg_{R'(G)}(x)\geq (\delta^+(D)+\delta(G))|V(R)|/|V(G)|-(2d+4\epsilon)|V(R)|.$$
\end{lemma}
\begin{proof}
The oriented graph $R$  is obtained by removing some arcs from $R'(D)$ as follows. Suppose $i, j\in V(R')$ are such that  $ij\in E(R'(D))$ and $\{i,j\}\notin E(R'(G))$. { Then add exactly one of $ij$, $ji$ to $R$. Specifically, add $ij$ with probability
$d_{D'}(V_i, V_j)/({d_{D'}(V_i, V_j)+d_{D'}(V_j, V_i)}),$ 
choices made independently for every $i$ and $j$.} Note that in the case $ij\in E(R'(D))$ but $ji\notin E(R'(D))$, the probability of adding $ij$ is one.  Then $R$ is an oriented subgraph of $R'$, and $R$ and $R'(G)$ are edge-disjoint. Let $X_i=deg_{R'(G)}(i)+ deg^+_{R}(i).$ Let $D'$ be the spanning sub-digraph of $D$, $G'$ the spanning subgraph of $G$ from Lemma \ref{deg-form}. Then $D'$ is an oriented graph and $D'$ and $G'$ are edge-disjoint. We have
\begin{align*}
	\E(X_i) & = deg_{R'(G)}(i)+ \sum_{j\in N^+_{R'(D)}(i)\setminus N_{R'(G)}(i)}\frac{d_{D'}(V_i, V_j)}{d_{D'}(V_i, V_j)+d_{D'}(V_j, V_i)} \\
		& \geq deg_{R'(G)}(i)+ \sum_{j\in N^+_{R'(D)}(i)\setminus N_{R'(G)}(i)}d_{D'}(V_i,V_j) \\
		& =\frac{1}{s^2}\left(\sum_{k\in N_{R'(G)}(i)}|V_i||V_k|+ \sum_{j\in N^+_{R'(D)}(i)\setminus N_{R'(G)}(i)}|E(D'[V_i,V_j])|\right). 
\end{align*}
For every $x\in V_i$ {and every $k\in N_{R'(G)}(i)$}, $|N^+_{D'}(x)\cap V_k|+|N_{G'}(x)\cap V_k|\leq |V_k|$ as $D'$ and $G'$ are edge-disjoint. Thus
\begin{align*}
	\E(X_i) & \geq \frac{1}{s^2}\sum_{x\in V_i}(d^+_{D'}(x)+d_{G'}(x)-|V_0|) \\
		& \geq \frac{1}{s^2}\sum_{x\in V_i}(d^+_{D}(x)-(d+\epsilon)|V(G)|+d_{G}(x)-(d+\epsilon)|V(G)|-\epsilon|V(G)|) \\
		& \geq \frac{1}{s}(\delta^+(D)+\delta(G) - (2d+3\epsilon)|V(G)|) \\
		& \geq (\delta^+(D)+\delta(G))|V(R)|/|V(G)|-(2d+3\epsilon)|V(R)|.
\end{align*}
Note that $X_i$ is a sum of independent Poisson trials (with some of the probabilities equal to one) and 
 we may assume $(\delta^+(D)+\delta(G))|V(R)|/|V(G)|-(2d+3\epsilon)|V(R)|\geq \epsilon |V(R)|,$ as otherwise the result holds trivially. 
 We have
\[	Pr(X_i<(\delta^+(D)+\delta(G))|V(R)|/|V(G)|-(2d+4\epsilon)|V(R)|)\leq Pr(X_i<\E(X_i)-\epsilon |V(R)|),\] 
and $\E(X_i)-\epsilon |V(R)|=(1-\xi)\E(X_i)$ for $\xi= \epsilon |V(R)|/\E(X_i)$. 
Then $\epsilon\leq \xi\leq 1$ since $\E(X_i)\ge \epsilon |V(R)|$ and $\E(X_i)\le |V(R)|$.
By Chernoff bound, 
 \[ Pr(X_i<\E(X_i)-\epsilon |V(R)|)<\exp(-\xi^2\E(X_i)/2)\leq \exp(-\epsilon^2 |V(R)|/2).\] 
 Therefore, from the union bound, 
 \begin{align*}
 	& Pr(X_i\ge (\delta^+(D)+\delta(G))|V(R)|/|V(G)|-(2d+4\epsilon)|V(R)|,\ \forall i\in V(R)) \\
		\ge & 1-\sum_{i\in V(R)} \exp(-\epsilon^2 |V(R)|/2) = 1- |V(R)| \exp(-\epsilon^2 |V(R)|/2)
 \end{align*}
  and the right-hand side is positive for large enough $|V(R)|=m$.
 Consequently, there is a spanning oriented subgraph $R$ of $R'$ such that for every $i\in V(R)$, $$X_i\geq (\delta^+(D)+\delta(G))|V(R)|/|V(G)|-(2d+4\epsilon)|V(R)|.$$
\end{proof}

\section{The Non-extremal Case}\label{non-section}

We will now show how to use the regularity lemma from the previous section to prove rainbow connectivity in the non-extremal case.

We will recall the notation first. Let $(G,c)$ be an edge-colored graph on $n$ vertices such that $\delta^c(G)\geq \frac{n}{2}$  and let $D=(V,E)$ be the digraph $D_G$ from Definition \ref{defnD}, and let $G^*=(V, E^*)$ be the graph on $V$ such that $\{u,v\}\in E^*$ if both $uv$ and $vu$ are in $E$ (Definition \ref{defG*}). Let $D^*=(V, F^*)$ be such that $uv\in F^*$ when $uv\in E$ but $vu\notin E$. Then $D^*$ is an oriented graph, and $D^*$ and $G^*$ are edge-disjoint. In addition, by Fact \ref{fact-outdeg}, \begin{equation}\label{delta-bd}\delta^+(D^*)+\delta(G^*)\geq n/2-\sqrt{n}.\end{equation} 

Let $0< \epsilon \ll d \ll 1$ be constants to be determined later and let $n$ be sufficiently large. 
{By Lemma \ref{cluster-graph-lem} with $D^*$ and $G^*$ in the place of $D$ and $G$, respectively,
there is a partition $V_0, \dots, V_t$ of $V(G^*)=V(D^*)$, a reduced oriented graph $R(D^*)$ in the place of $R$, and a reduced graph $R(G^*)$ in the place of $R'(G)$,} such that  $R(D^*)$ and $R(G^*)$ are edge-disjoint, and for every $j\in [t]$, by (\ref{delta-bd}), 
\begin{equation}\label{eq-outdeg}
deg^+_{R(D^*)}(j)+deg_{R(G^*)}(j)\geq (\delta^+(D^*)+\delta(G^*))t/n- (2d+4\epsilon)t\geq {(1/2-2d-5\epsilon)}t.
\end{equation}
Let $s:=|V_1|=\cdots =|V_t|$ and recall that $|V_0|<\epsilon n.$
\begin{lemma}\label{path-cluster-lem}
If there is $j\in [t]$ such that $deg^-_{R(D^*)}(j)+deg_{R(G^*)}(j)\geq (1/2+3d)t$, then for every $x, y\in V(G)$ with $x\neq y$ there is a rainbow $x,y$-path in $G$ of length six.
\end{lemma}

\begin{proof}
Suppose $deg^-_{R(D^*)}(j)+deg_{R(G^*)}(j)\geq (1/2+3d)t$. Then, by (\ref{eq-outdeg}), $deg_{R(G^*)}(j)\geq 1$. Let $i\in [t]\setminus \{j\}$ be such that $\{i,j\}\in E(R(G^*)).$
For $z\in \{x,y\}$, let $i_z\in N^-_{R(D^*)}(j)\cup N_{R(G^*)}(j)\setminus\{i\}$ be such that $|N^+_{D}(z)\cap V_{i_z}|\geq 2\epsilon s$. 
Note that such $i_z$ exists, as otherwise, 
\begin{align*}
	d^+_D(z) & = d^+_{D^*}(z)+d_{G^*}(z) \\
	 & < (1/2-3d)t\cdot s + ( (1/2+3d)t -1)\cdot 2\epsilon s + (d+\epsilon)|V|,	
\end{align*}
where the first two terms are the edges from $z$ in $D^*$ and $G^*$ after deleting $V_i$, and the third term is the possible degrees deleted in the regularity lemma.
Thus 
\begin{align*}
	d^+_D(z)  & < (1/2-3d)n + (1/2 + 3d)2\epsilon n + (d+\epsilon) n  \\
		 & < (1/2 -2d + 2\epsilon +6d\epsilon)n,
\end{align*}
which is a contradiction for $\epsilon< d/2< 1/6$.

Now we find a rainbow $x,y$-path $P$ in $G$ of length six greedily. 
Let $C_P$ denote the colors that occur on $P$, and initially set $C_P := \emptyset$. 
First, since $(V_{i_x}, V_j)$ and $(V_j, V_i)$ are $\epsilon$-regular pairs, we can take $u_x\in V_{i_x}$, $w_x\in V_j$ such that $xu_xw_x$ is a rainbow path and $|N_{G^*}(w_x)\cap V_i|\geq (d-\epsilon)|V_i|$. Add $c(xu_x)$ and $c(u_xw_x)$ to $C_P$.
Then we take $u_y\in V_{i_y}\setminus \{x, u_x, w_x\}$ such that $c(yu_y)\notin C_P$ and $|N^+_{D^*}(u_y)\cap V_j|\geq (d-\epsilon)|V_j|$, and add $c(yu_y)$ to $C_P$.

Note that, by construction, the edges in $E(G^*)$ are properly colored in $(G,c)$, and the edges in $E(D^*)$ from a fixed vertex have distinct colors. Therefore, all edges between $w_x$ and  $N_{G^*}(w_x)\cap V_i$ have distinct colors, and all edges between $u_y$ and $N^+_{D^*}(u_y)\cap V_j$ have distinct colors. Let 
\begin{align*}
& W_x := \{v\in N_{G^*}(w_x)\cap V_i : c(vw_x) \notin C_P\}, \quad \text{and}  \\
& W_y := \{v\in N^+_{D^*}(u_y)\cap V_j : c(vu_y) \notin C_P\}, 
\end{align*}
then 
\begin{align*}
|W_x|\ge |N_{G^*}(w_x)\cap V_i|-3\ge \epsilon \cdot s \quad \text{ and } \quad |W_y|\ge |N^+_{D^*}(u_y)\cap V_j|-3 \ge \epsilon \cdot s.
\end{align*}

Thus, there are at least $(d-\epsilon)|W_x||W_y|$ properly colored edges in $G^*$ between $W_x$ and $W_y$. 
We take $z\in W_x\setminus \{x,u_x,y,u_y\}$ and $w_y\in W_y\setminus \{x,u_x, w_x,y\}$ such that $zw_y$ is an edge and $c(zw_y) \notin C_P$, and $w_xzw_yu_y$ is a rainbow path. Hence, $xu_xw_xzw_yu_yy$ is a rainbow $x,y$-path of length six.
\end{proof}

In view of Lemma \ref{path-cluster-lem}, we may assume that for every $j\in [t]$, 
\begin{equation}\label{max-indeg-eq}
    deg^-_{R(D^*)}(j)+deg_{R(G^*)}(j)<(1/2+3d)t.
\end{equation}
Let $A=\{j\in [t] : deg^-_{R(D^*)}(j)+deg_{R(G^*)}(j)<(1/2-\sqrt{d})t\}$. By (\ref{eq-outdeg}) and (\ref{max-indeg-eq}), we have
${(1/2-2d-5\epsilon)}t^2< (1/2-\sqrt{d})|A|t+ (t-|A|)(1/2+3d)t$, and so \begin{equation}\label{eq-non-ext-A}{|A|< \frac{5d+5\epsilon}{\sqrt{d}+3d}t<5\sqrt{d}t,}\end{equation}
as long as $\epsilon \leq d^{3/2}.$

\begin{lemma}\label{non-ext-lem}
    Let $\beta>0$ be such that $1/n\ll \beta \ll 1$. If $(G,c)$ is an edge-colored graph on $n$ vertices that is not $\beta$-extremal of type 1 and not $\beta$-extremal of type 2, then for every $x,y\in V(G)$ with $x\neq y$, there is a rainbow $x,y$-path of length at most eight.
\end{lemma}
\begin{proof}
Let $V:=V(G)$. Since $(G,c)$ is not $\beta$-extremal of type 2, $|E(G^*)|\geq \beta |V|^2$. Consequently, since $0<\epsilon \ll d\ll \beta$ there is $j^*\in [t]$ such that $deg_{R(G^*)}(j^*)>0$ and $j^*\notin A.$ Indeed, otherwise, {using (\ref{eq-non-ext-A}),}
\[ \sum_{x\in V(G)}d_{G^*}(x) \leq (|A|\cdot s+|V_0|)|V|+(d+\epsilon)|V|^2/2<(6\sqrt{d}+2\epsilon)|V|^2<\beta |V|^2.\]

Let $i^*\in [t]$ be such that $\{i^*,j^*\}\in E(R(G^*))$  and let $x,y \in V$ with $x\neq y$. We construct a rainbow $x,y$-path $P$ in $G$ of length at most eight greedily. 
Let $C_P$ denote the colors that occur on $P$, and initially set $C_P := \emptyset$. 

\medskip

We first construct
vertex-disjoint paths $P_1= x\dots w_x$ and $P_2=y\dots w_y$, each of length at most two, so that $P_1\cup P_2$ is rainbow and $|N^+_{D}(w_x)\cap V_{j^*}|\geq \epsilon s$, $|N^+_{D}(w_y)\cap V_{j^*}|\geq \epsilon s$. 

For $z\in \{x,y\}$, define 
$$N^+_{R}(z) : =\{j\in [t] : |N^+_{D}(z)\cap V_j|\geq \epsilon |V_j|\}.$$
As in the proof of the previous lemma, we can argue that 
\begin{equation}\label{eq-non-ext-N+R}|N^+_R(z)|\geq (1/2-3\epsilon)t.
\end{equation}
Indeed, otherwise
$d^+_D(z)< (1/2-3\epsilon)t\cdot s+ (1/2+3\epsilon)t \cdot \epsilon s+ |V_0|,$
which is not possible.
Let $$N^-(j^*) := N^-_{R(D^*)}(j^*)\cup N_{R(G^*)}(j^*),$$ {and note that since $j^*\notin A$,
\begin{equation}\label{eq-non-extN^-}
|N^-(j^*)|\geq (1/2-\sqrt{d})t.
\end{equation}}
For $z=x$, and then $z=y$, we consider the following cases.

If $N^+_R(z)\cap N^-(j^*)\neq \emptyset$, then let $j\in N^+_R(z)\cap N^-(j^*)$, and let $w_z\in V_j$ be such that $|N^+_{D'}(w_z)\cap V_{j^*}|\geq \epsilon s$, and $c(zw_z) \not\in C_P$. Add $zw_z$ to $P$ and add $c(zw_z)$ to $C_P$. 

Now suppose $N^+_R(z)\cap N^-(j^*)= \emptyset$.
If there is $j\in N^+_R(z)$ and $i\in N^-(j^*)$ such that $ji\in E(R(D^*))$ or $\{i,j\}\in E(R(G^*))$, then let $u_z\in V_j$, $w_z\in V_i$ be such that $zu_zw_z$ is rainbow if $z=x$, and $zu_zw_z\cup P$ is rainbow if $z=y$, respectively, and $|N^+_{D}(w_z)\cap V_{j^*}|\geq \epsilon |V_{j^*}|$. Add $zu_zw_z$ to $P$ and add $c(zu_z), c(u_zw_z)$ to $C_P$. 

Otherwise, no such pair $(j,i)$ exists. Then 
all the out-going edges from $N^+_R(z)$ in $R(D^*)$ and edges in $R(G^*)$ incident to cluster in $N^+_R(z)$ land in $V(R(D^*))\setminus N^-(j^*).$ {In other words,}
\begin{align*}
	 V(R(D^*))\setminus N^-(j^*) 
	\supseteq  \{k : \exists j\in N^+_R(z),\  jk\in E(R(D^*)) \text{ or }\{k,j\}\in E(R(G^*))\}.
\end{align*}
{In addition,} $N^+_R(z) \subseteq V(R(D^*))\setminus N^-(j^*)$ as $N^+_R(z)\cap N^-(j^*)= \emptyset$.
By (\ref{eq-non-extN^-}), $|N^-(j^*)|\geq (1/2-\sqrt{d})t, $ and so
\begin{equation}\label{eq-non-ext-compu}
|V(R(D^*))\setminus N^-(j^*)| \le (1/2+\sqrt{d})t.\end{equation} 
{We have the following claim, the proof of which is deferred until the end of the section.}
{
\begin{claim}\label{non-ext-claim1}
All but at most $10\sqrt{d}t$ clusters $j\in N^+_R(z)$ satisfy $deg_{R(G^*)}(j)>0$ and $j\notin A.$\end{claim}}

Let $\tilde{N}^+_R(z)$ be the set of clusters $j\in N^+_R(z)$ that satisfy $deg_{R(G^*)}(j)>0$ and $j\notin A.$ By Claim \ref{non-ext-claim1} and
(\ref{eq-non-ext-N+R}), $|\tilde{N}^+_R(z)|\geq (1/2- 11\sqrt{d})t$. 

Let $u \in \{x,y\}\setminus \{z\}$.
If $N^+_R(u)\cap \tilde{N}^+_R(z) \neq \emptyset$, 
then replace $j^*$ by a cluster from $N^+_R(u)\cap \tilde{N}^+_R(z)$, update $i^*$ so that $\{i^*,j^*\}\in E(R(G^*))$, and let $w_x := x$, $w_y := y$, $P_1 := x$, $P_2 := y$, and update $P$ to $P_1\cup P_2$ and set $C_P:=\emptyset$.

Otherwise, {keep the original $j^*$ and note that by (\ref{eq-non-extN^-}) and (\ref{eq-non-ext-N+R}),} $$|N^+_R(u)\cap N^-(j^*) |\geq (1/2-13\sqrt{d})t.$$
Since $(G,c)$ is not $\beta$-extremal of type 1 there exist $i \in \tilde{N}^+_R(z)$ and $j \in N^+_R(u)\cap N^-(j^*) $ such that 
$ij\in E(R(D^*))$
or $ji\in E(R(D^*))$. 
Since we are in the case where there is no edge from $N^+_R(z)$ to $N^-(j^*) $, we have $ji\in E(R(D^*))$. 
Now replace $j^*$ by this cluster $i$, update $i^*$ so that $\{i^*,j^*\}\in E(R(G^*))$, and let $w_z := z$, $w_u \in V_j$ be such that $|N^+_{D}(w_u)\cap V_{j^*}|\geq \epsilon s$, 
$\{P_1, P_2\} := \{z, uw_u\}$, and update $P$ by $P_1\cup P_2$ and $C_P$ by $\{c(uw_u)\}$.

\medskip

Thus in all cases, we have constructed vertex-disjoint path $P_1$, $P_2$ where $P_1$ is an $x,w_x$-path and $P_2$ a $y,w_y$-path. Now we construct $P$ by connecting $w_x, w_y$ with a rainbow path of length four without using any color from $C_P$. 

We have $|N^+_{D}(w_x)\cap V_{j^*}|\geq \epsilon s$, 
$|\{u\in N^+_{D}(w_x)\cap V_{j^*} : c(w_xu)\in C_P\}|\le 4$, and 
\begin{align*}
& |\{u\in N^+_{D}(w_x)\cap V_{j^*}: |N_{G^*}(u)\cap V_{i^*}|{\geq} \epsilon s\}|  \\
\ge \ & (|E_{G^*}(N^+_{D}(w_x)\cap V_{j^*}, V_{i^*})| - |N^+_{D}(w_x)\cap V_{j^*}|\epsilon s)/s \ge \epsilon^2 s.
\end{align*}
Therefore, we may take $z_1 \in N^+_{D}(w_x)\cap V_{j^*}$ such that $c(w_xz_1) \not\in C_P$ and $|N_{G^*}(z_1)\cap V_{i^*}|\geq \epsilon s$.
Add $w_xz_1$ to $P$ and add $c(w_xz_1)$ to $C_P$.

Since $|N^+_{D}(w_y)\cap V_{j^*}|\geq \epsilon s$ and $|N_{G^*}(z_1)\cap V_{i^*}|\geq \epsilon s$, there are at least $(d-\epsilon)|N^+_{D}(w_y)\cap V_{j^*}|\cdot |N_{G^*}(z_1)\cap V_{i^*}|\ge (d-\epsilon)\epsilon^2 s^2$ properly colored edges in $G^*$ between $N^+_{D}(w_y)\cap V_{j^*}$ and $N_{G^*}(z_1)\cap V_{i^*}$. Let $E^*$ denote the set of these edges. 
Now for each vertex $u\in N^+_{D}(w_y)\cap V_{j^*}$, if $c(w_yu)\in C_P$, we remove all edges incident to $u$ from $E^*$; otherwise,
we remove the edges $uv\in E^*$ such that $c(uv) \in C_P\cup \{c(w_yu)\}$ or $c(z_1v) \in C_P\cup \{c(w_yu)\}$. Thus, in total, we have removed at most $|C_P|+2(|C_P|+1)s\le 13s$ edges, and hence we may take an edge remaining in $E^*$, say $z_2z_3$, where $z_2\in N^+_{D}(w_y)\cap V_{j^*}$ and $z_3 \in N_{G^*}(z_1)\cap V_{i^*}$. 

Note that $c(z_1z_3)\ne c(z_2z_3)$ as $G^*$ is properly colored, from the choice of $z_2,z_3$, $xP_1w_xz_1z_3z_2w_yP_2y$ is a rainbow path of length at most 8. 
\end{proof}

\begin{proof}[Proof of Claim \ref{non-ext-claim1}]
Note that by Lemma~\ref{cluster-graph-lem} {$R(D^*)$ is an oriented graphs, and so} for all $i,j \in [t]$ at most one of $ij, ji$ is in $E(R(D^*))$, 
and $R(D^*)$ and $R(G^*)$ are edge disjoint.  
Thus, {using (\ref{eq-non-ext-compu}),}
\begin{align*}
	& \sum_{j\in N^+_R(z)} deg^+_{R(D^*)}(j)+deg_{R(G^*)}(j)  \\
	\le \ & |E_{D^*}(N^+_R(z), V(R(D^*))\setminus N^-(j^*))| + 2|E_{G^*}(N^+_R(z), V(R(D^*))\setminus N^-(j^*))| \\
	\le \ & |E_{D^*}[V(R(D^*))\setminus N^-(j^*))]| +2 |E_{G^*}[V(R(D^*))\setminus N^-(j^*))]| \\
	\le \ & ((1/2+\sqrt{d})t)^2/2 + |E_{G^*}[V(R(D^*))\setminus N^-(j^*))]|.
\end{align*}

On the other hand, by (\ref{eq-outdeg}), $deg^+_{R(D^*)}(j)+deg_{R(G^*)}(j)\geq {(1/2-2d-5\epsilon)}t$ for all $j\in [t]$. So, by  (\ref{eq-non-ext-N+R}),
\begin{align*}
	 \sum_{j\in N^+_R(z)} deg^+_{R(D^*)}(j)+deg_{R(G^*)}(j)  
	\ge  (1/2-3\epsilon)t \cdot (1/2-2d-5\epsilon)t.
\end{align*}
Thus $|E_{G^*}[V(R(D^*))\setminus N^-(j^*))]| \ge (1/4 - 2\sqrt d)t^2/2$, and hence
all but at most $10\sqrt{d}t$ clusters $j\in N^+_R(z)$ satisfy $deg_{R(G^*)}(j)>0$ and $j\notin A.$ 
\end{proof}

\section{Extremal Cases}\label{ext-section}
In this section, we will discuss extremal cases. Recall that $(G,c)$ is an edge-colored graph on $n$ vertices such that $\delta^c(G)\geq \frac{n}{2}$ and subject to this $E(G)$ is minimal. Let $D:=D_G$ be the digraph from Definition \ref{defnD}, and let $G^*$ be the graph from Definition \ref{defG*}.

\subsection{Extremal Case 1}
This case is basic. Let $\beta>$ be such that $ \sqrt{\beta}\leq 1/320$, let  $n$ be sufficiently large so that \begin{equation}\label{n-ex-case1}1/\sqrt{n}<\beta,\end{equation} and suppose $G$ is $\beta$-extremal of type 1. Then there exists a partition $\{V_1, V_2\}$ of $V(G)$ such that for $i=1,2$, $|V_i|\geq (1/2-\beta)n$ and the number of arcs in $D$ between $V_1$ and $V_2$ is at most $\beta n^2.$ Let $V_i':=\{v\in V_i: |N_{G^*}(v)\cap V_i|\geq (1/2-\sqrt{\beta})n\}$. We have the following bound.
\begin{fact}
 $|V_i'|\geq (1/2-7\sqrt{\beta})n.$
\end{fact}
\begin{proof} 
	By Fact \ref{fact-outdeg},  
	\begin{equation}\label{eq1-fact5.1}
		\sum_{v\in V_i} d_D^+(v)\geq (1/2-1/\sqrt{n})n |V_i|.
	\end{equation}
	At the same time, for $v\in V_i\setminus V_i'$, $|N_{G^*}(v)\cap V_i|< (1/2-\sqrt{\beta})n <|V_i|-(\sqrt{\beta}-\beta)n,$ 
	and so $|N^+_D(v)\cap V_i|+|N^-_D(v)\cap V_i|< 2|V_i|-(\sqrt{\beta}-\beta)n$. Therefore,
	\[
		2\sum_{v\in V_i}|N^+_D(v)\cap V_i|= \sum_{v\in V_i}(|N^+_D(v)\cap V_i|+|N^-_D(v)\cap V_i|)< 2|V_i|^2 -(\sqrt{\beta}-\beta)n|V_i\setminus V_i'|.
	\]
	Consequently,
	\begin{equation}\label{eq2-fact5.1}
		\sum_{v\in V_i} d_D^+(v)< \beta n^2+|V_i|^2 - (\sqrt{\beta} -\beta)n|V_i\setminus V_i'|/2.
	\end{equation} 
	By (\ref{eq1-fact5.1}), (\ref{eq2-fact5.1}) and (\ref{n-ex-case1}),
	$|V_i'|> \frac{\sqrt{\beta} -13\beta}{\sqrt{\beta} -\beta}|V_i|,$ and so $|V_i'|\geq (1/2-7\sqrt{\beta})n$ with some room to spare.
\end{proof}

\begin{lemma}\label{ext-type1-lem}
    Let $\beta>0$ and $n\in \mathbb Z^+$ be such that $1/n\ll \beta \ll 1$. If $(G,c)$ is an edge-colored graph on $n$ vertices such that $\delta^c(G)\geq \frac{n}{2}$ which is $\beta$-extremal of type 1, then for every two distinct vertices $x,y\in V(G)$ there is a rainbow $x,y$-path in $G$ of length at most five.
\end{lemma}
\begin{proof} Let $W:=V(G)\setminus (V_1'\cup V_2')$.  Note that $|W|\leq 14\sqrt{\beta}n$ and so, for every $w\in W$ there is $i=1,2$ such that $|N^c(w)\cap V_i'|\geq n/5.$

Let $W_1:= \{w\in W:|N^c(w)\cap V_1'|\geq n/5\}$, and $W_2:=W\setminus W_1$. Let $M$ be a set of colors, and let $S\subseteq V_1'\cup V_2'$ be a set of vertices such that $|M\cup S|\leq 10$. 
  
  For any $u, u'\in V_i'\setminus S$, $|N_{G^*}(u)\cap N_{G^*}(u')\cap V_i'|> (1/2-{ 10}\sqrt{\beta})n >n/3,$ and so there is $z\in (N_{G^*}(u)\cap N_{G^*}(u')\cap V_i')\setminus S$ such that $uzu'$ is a rainbow path that does not contain colors from $M$. In addition, for every $w\in W_i\cup V_i$ there is a vertex $u\in V_i'\setminus  S$ such that $wu$ is an edge and $c(wv_i)\notin M$. 
  
  Let $x,y$ be two arbitrary vertices. First suppose that for some $i=1,2$,  $x,y\in V_i'\cup W_i$.  There is $u\in V_i'$ such that $xu$ is an edge. Add $u$ to $S$ and add $c(xu)$ to $M$. Then there $u'\in V_i\setminus S$ such that  $yu'$ is an edge such that $c(yu')\notin M$. Add  $c(yu)$ to $M$ and $y$ to $S$. Finally, there $z\in V_i' \setminus S$ such that $uzu'$ is a path of length two that does not contain colors from $M$. Then $xuzu'y$ is a rainbow $x,y$-path of length four. 
  
  Now suppose $x\in V_1'\cup W_1$ and $y\in V_2'\cup W_2$.
  Without loss of generality, $|V_1'\cup W_1|\leq n/2$, and so for every $v\in V_1'\cup W_1$ there is $w\in V_2'\cup W_2$ such that $vw\in E(G)$. Let $z\in V_2'\cup W_2$ be a neighbor of $x$ in $G$. There is a rainbow  $y,z$ path $P$ of length four that avoids $c(xz)$. Then $xzPy$ is a rainbow $x,y$ path of length five.
 \end{proof}

\subsection{Extremal Case 2} We will now consider the extremal case  of type 2. Suppose $\beta>0$ is sufficiently small and $|E(G^*)|\leq \beta n^2$. In addition, suppose $n$ is sufficiently large with respect to $1/\beta$.  
Let $U:=\{v\in V(G): d^-_{D}(v)\geq (1/2-\sqrt{\beta}) n\}$ and let $W:= V(G)\setminus U.$ 
\begin{fact} \label{factU-ext}$|U|\geq (1-3\sqrt{\beta}/2)n.$
\end{fact}
\begin{proof} By Fact \ref{fact-outdeg}, $\delta^+(D)> n/2- \sqrt{n}.$ We have
$$\frac{n^2}{2}-n^{3/2}< \sum d^+_D(v)=\sum d^-_D(v)\leq $$
$$|U|(n/2+\sqrt{n}) +(n-|U|)(1/2-\sqrt{\beta})n +\beta n^2.$$
This gives $|U|\geq \frac{\sqrt{\beta}n-\beta n-\sqrt{n}}{\sqrt{\beta}n+\sqrt{n}}n> (1-3\sqrt{\beta}/2)n,$ as $1/\beta \ll n.$
\end{proof}

\begin{fact} \label{factXY-ext} Let $X,Y\subseteq V(G)$ be such that $|X|\geq (1/2-2\sqrt{\beta})n$ and $|Y|\geq (1/2-2\sqrt{\beta})n$. Then there are at least $\beta n^2$ arcs in $D$ from $X$ to $Y$.
\end{fact}
\begin{proof}
First suppose $|X\cap Y|\geq 3\sqrt{\beta}n$. If the number of arcs in $D[X\cap Y]$ is less than $\beta n^2$, then, since $|E(G^*)|\leq \beta n^2$, there are less than $$n^2/2 +|E(D[X\cap Y])|+ \beta n^2 -{3\sqrt{\beta} n \choose 2}\leq n^2/2 +2\beta n^2 -{3\sqrt{\beta} n \choose 2}$$ arcs in $D$. 
This contradicts $\delta^+(D)>n/2-\sqrt{n}.$ Now suppose $|X\cap Y|<3\sqrt{\beta}n$. Then $|V(D)\setminus (X\cup Y)|<7\sqrt{\beta}n$, and $$(1/2+2\sqrt{\beta})n\geq |X\setminus Y|\geq (1/2-5\sqrt{\beta})n.$$ If there are less than $\beta n^2$ arcs from $X$ to $Y$, then
$$|X\setminus Y|(n/2-\sqrt{n}) < \sum_{x\in X\setminus Y} d_D^+(x)< 2\beta n^2 +7\sqrt{\beta}n|X\setminus Y| +|X\setminus Y|^2/2,$$
which contradicts $1/\beta \ll n$.
\end{proof}
Recall that if $wv\in E(D)$ and $\alpha=c(wv)$ then there are at most $\sqrt{n}$ vertices $x\in N^-_D(v)$ such that $c(xw)=\alpha.$ The following lemma will be used repeatedly in the rest of the argument.
\begin{lemma}\label{lemmaxu-ext}
Let $M$ be a set of colors and let $S$ be a set of vertices and suppose $|M\cup S|\leq \sqrt{n}$. Let $u\in U\setminus S$ be such that there are at least $|M|/\beta +|S| +1$ vertices $z\in N^-_D(u)\cap U$ with $c(zu)\notin M$. Then for every $x\in V(G)\setminus (S\cup \{u\})$ there is a rainbow $x,u$-path in $G$ of length four that avoids colors from $M$ and vertices from $S.$  
\end{lemma}
\begin{proof} Let $Z$ denote the set of vertices $z\in (N^-_D(u)\setminus S)\cap U$ such that $c(zu)\notin M$ 
and let $Z'$ be the set of vertices $w\in Z$ such that for some $\alpha \in M$, $\alpha$ appears on at least $\beta n$ edges $vw$. 
Since for every $\alpha$, $F_\alpha$ is a star forest, $|Z'|\leq |M|/\beta$.  Let $z\in Z\setminus Z'$  and let $Y$ be obtained from $N^-_D(z) \setminus S$ by deleting vertices $v$ such that $c(vz)\in M$ or $c(vz)=c(zu)$. Since $z\in Z'$, there are less than $\beta n$ vertices $v\in N^-_D(z)$ such that $c(vz)\in M$, and by the definition of $D$, there are at most $\sqrt{n}$ vertices $v\in N^-_D(z)$ such that $c(vz)= c(zu)$. Consequently,  
\begin{equation}\label{eq-Ylem5.5}
|Y|\geq (1/2-2\sqrt{\beta})n.
\end{equation}
Let $x\in V(G)\setminus (S\cup \{u\})$ and let $X$ be obtained from 
$N^+_D(x)\setminus S$ by deleting vertices $v$ such that $c(xv)\in M\cup \{c(zu)\}.$ Since for every color $\alpha$ there is at most one vertex $v\in N^+_D(x)$ with $c(xw)=\alpha$, \begin{equation}\label{eq-Xlem5.5}|X|\geq (1/2-2\sqrt{\beta})n.\end{equation} 
By (\ref{eq-Ylem5.5}), (\ref{eq-Xlem5.5}) and Fact \ref{factXY-ext}, there are at least $\beta n^2$ arcs from $X$ to $Y$. We delete three types of arcs from $X$ to $Y$: 
\begin{itemize}
    \item[(i)] arcs of color  $\alpha \in \{c(zu)\} \cup M$,
    \item[(ii)]for every $v\in Y$ arcs $wv$ such that $c(wv)=c(vz)$, 
    \item[(iii)] for every $w\in X$ at most one arc $wv$ such that $c(wv)=c(xw).$
\end{itemize}
The total number of deleted arcs is less than $$(|M|+1)n +n\sqrt{n} +n\leq  (2\sqrt{n}+2)n<
\beta n^2/2.$$ Therefore, there are at least $\beta n^2/2$ arcs that remain. Consequently, there are two paths $P_1=xw_1vzu$, $P_2=xw_2vzu$ such that $w_1,w_2\in X$, $v\in Y$ and $w_1v, w_2v$ were not deleted. Either $c(xw_1)\neq c(vz)$ or $c(xw_2)\neq c(vz)$, and so one of $P_1, P_2$ is a rainbow path that avoids $M$ and $S$.
\end{proof}

Recall that $E_D(V_1, V_2)$ denotes the set of arcs in $D$ from $V_1$ to $V_2.$ If $V_2$ is a singleton $\{u\}$, we simplify it as $E(V_1, u)$ when there is no confusion.

\begin{defn} Let $u\in U$. We say that $u$ is a rainbow link if there exist two disjoint subsets of $N^-_D(u)$, $V_1, V_2$,  such that $|V_i|\geq 2\sqrt{n}$ and $c(E(V_1, u))\cap c(E(V_2, u))=\emptyset.$
\end{defn}

\begin{lemma}\label{lem1-ext}
If there exists a rainbow link $u\in U$, then for every two distinct vertices $x,y\in V(G)$ there is a rainbow $x,y$-path of length at most eight that contains $u$. 
\end{lemma}
\begin{proof}  We may assume $x, y \in V(G)\setminus \{u\}$.\\
{\it Case 1.} There exists $V_2' \subseteq V_2$ of size at most $\sqrt{n}$ such that $|c(E(V_2\setminus V_2', u))|\leq 4.$

Let $M:=c(E(V_2\setminus V_2', u)).$ Since $c(E(V_1, v))$ does not contain colors from $M$ and $|V_1|\geq 2\sqrt{n}$, by Lemma \ref{lemmaxu-ext} with $S:=\emptyset$, there is a rainbow $x,u$-path $P_1$ of length four that avoids colors from $M.$
 We have $|V_2\setminus V_2'|\geq \sqrt{n}$, and so by Lemma \ref{lemmaxu-ext} with $M:=c(E(P_1))$ and $S:=V(P_1)\setminus \{u\}$ there is a rainbow $y,u$-path $P_2$ of length four that avoids colors on $P_1$ and vertices from $V(P_1)$. Then $P_1\cup P_2$ is a rainbow $x,y$-path of length eight.\\
{\it Case 2.} For every $V_2'\subseteq V_2$ of size at most $\sqrt{n}$, $|c(E(V_2\setminus V_2', u))|> 4.$

Apply Lemma \ref{lemmaxu-ext} with $S=M=\emptyset$ to get a rainbow $x,u$-path $P_1$ of length four. Let $V_2'':= \{z\in V_2: c(zu)\in c(E(P_1))\}$ and let $V_2':= V_2\setminus V_2''.$ Suppose $|V_2'|\leq \sqrt{n}$. Then $V_2'$ is a subset of $V_2$ of size at most $\sqrt{n}$ such that $|c(E(V_2\setminus V_2', u))|\leq |c(E(P_1))|=4.$ Thus $|V_2'|\geq \sqrt{n}$ and by Lemma \ref{lemmaxu-ext} with $M:= c(E(P_1))$ and $S:=V(P_1)\setminus \{u\}$ there is a rainbow $y,u$-path $P_2$ that avoids colors from $c(E(P_1))$ and vertices from $V(P_1)\setminus \{u\}$.
\end{proof}

By Lemma \ref{lem1-ext} we may assume that there is no rainbow link $u\in U$. Consequently, for every $u\in U$, there is a color $c_u$ such that all but at most $2\sqrt{n}$ vertices $z\in N^-_D(u)\cap U$ satisfy $c(zu)=c_u$. 

\begin{lemma}\label{lem2-ext}
 \begin{itemize}
\item[(a)] Let $u\in U$ and let $\alpha \neq c_u$. For every $x\in V(G)$  there is a rainbow $x,u$-path of length at most four that avoids $\alpha$.
\item[(b)] For $w\in W$ if $uw$ is an arc in $D$ for some $u\in U$, then for every $x\in V(G)$ there is a rainbow $x,w$- path in $G$ of length at most five.
\end{itemize}
\end{lemma}
\begin{proof}Part (a) follows from Lemma \ref{lemmaxu-ext} with $S:=\emptyset$ and $M:=\{\alpha\}$. 
For (b), suppose $uw$ is an arc with $u\in U$ and let $\alpha:= c(uw)$. Note that $c_u\neq \alpha$ by the definition of $D$. Let $M:=\{\alpha\}$ and let $S:=\{w\}$. By Lemma \ref{lemmaxu-ext} there exists a rainbow $x,u$ path of length at most four that avoids $\alpha$ and does not contain $w$. 
\end{proof}

Let $\gamma:=\sqrt{\beta}/16$ and let $W':=\{w\in W: |N^-_D(w)\cap U|\geq \gamma n\}$.

\begin{lemma}\label{lem3-ext}Let $w\in W'$ and $v\in W\setminus W'$ be such that $wv$ is an arc in $D$. Then for every $x\in V(G)$ there is a rainbow $x,v$-path of length at most six.
\end{lemma}
\begin{proof} Let $\alpha:=c(wv).$ Since $w\in W'$ there are at least $\gamma n - \sqrt{n}> \gamma n/2$  vertices $u\in N^-_D(w)\cap U$ such that $c(uw)\neq \alpha.$ Since for every $u\in U$, $|c^{-1}(c_u)| \geq (1/2-2\sqrt{\beta})n$, there are at most two vertices $u\in N^-_D(w)\cap U$ such that $c_u= \alpha$. Let $u\in N^-_D(w)\cap U$ be such that $\alpha\notin \{c(uw), c_u\}.$
Now apply Lemma \ref{lemmaxu-ext} with $M:=\{\alpha, c(uw)\}$ and $S:= \{v,w\}$ to get a rainbow $x,u$-path $P$ of length at most four that avoids colors $\alpha$ and $c(uw)$ and does not contain $v$ or $w$. Then $Pwv$ is a rainbow $x,v$-path of length at most six. 
\end{proof}
\begin{lemma}\label{ext-case-W-W'}
If $|W\setminus W'|\leq 1$, then for any two vertices $x,y\in V(G)$ there is a rainbow $x,y$-path of length at most five. 
\end{lemma}
\begin{proof} Let $x,y\in V(G)$, and suppose $x\in U\cup W'$. Then $N^-_D(x)\cap U\neq \emptyset$ and so by Lemma \ref{lem2-ext}, there is a rainbow $y,x$-path of length at most five.
\end{proof}
For the rest of the argument, we shall assume that there exist two vertices in $G$ such that there is no rainbow path between them of length at most nine, and so, in particular in view of Lemma \ref{ext-case-W-W'},  $|W\setminus W'|\geq 2.$
\begin{defn}
Let $D'=(V(G),E')$ be the following digraph obtained from $D=(V(G), E)$. Let $E\subseteq E'$ and, in addition, for every $\alpha \in c(E(G))$ and $v\in V(G)$ if $\sqrt{n}< d_{F_\alpha}(v)$ then choose one vertex $w\in N_{F_{\alpha}}(v)$ and add $vw$ to $E'$. Moreover, if possible, choose $w$ to be in $U.$
\end{defn}
We have \begin{equation}\label{deltaD'}
\delta^+(D')=\delta^c(G)\geq \frac{n}{2}.
\end{equation}
We extend the definition of $G^*$ to account for the additional edges in $D'$ and set $G^*:= G^*_{D'}= (V(D'), F)$ where $\{u,v\}\in F$ when both $uv$ and $vu$ are in $D'.$

\begin{lemma}\label{lem4-ext} Let $r:=|E(G^*[U])|/|U|$.
\begin{enumerate}
\item[(a)] If $|W\setminus W'|\geq 3$, then $r>3/2$.
\item[(b)] If $|W\setminus W'|=2$ and $|W'|\geq 1$ then $r>3/2.$
\end{enumerate}
\end{lemma}
\begin{proof} 
We double count the in edges from $U$. 
\begin{align*}
    n|U|/2 & \le \sum_{v\in U} d^+(v)
     \le \binom{|U|}{2} + r|U| + (1/2-\sqrt{\beta}) n |W'| +\gamma n |W\setminus W'|, 
\end{align*}
where the first inequality is from (\ref{deltaD'}) and the second inequality is from the definition of $U$ and $W'$. Then
$$\frac{n}{2}< \frac{|U|-1}{2}+ (1/2 -\sqrt{\beta}) n|W'|/|U| +\gamma n|W\setminus W'|/|U| +r.$$
By Fact \ref{factU-ext}, $|U|\geq (1-3\sqrt{\beta}/2)n$. Therefore,
$n/|U| \leq 1/(1-3\sqrt{\beta}/2) \leq 1+2\sqrt{\beta}<2$ and $(1/2-\sqrt{\beta})/(1-3\sqrt{\beta}/2) <1/2-\sqrt{\beta}/4$. 
Consequently\begin{equation}\label{eq-lem5.12}\frac{n}{2} < \frac{|U|-1}{2} +(1/2 -\sqrt{\beta}/{4})|W'| +2\gamma|W\setminus W'| +r.\end{equation}
Since $n=|U|+|W'|+|W\setminus W'|$, by (\ref{eq-lem5.12}),
$$\sqrt{\beta} |W'|/4  +(1/2-2\gamma) |W\setminus W'| +\frac{1}{2} <r.$$
If $|W\setminus W'|\geq 3$, then
$r>2 -6\gamma>3/2.$
If $|W \setminus W'|=2$ and $| W'|\geq 1$, then 
$r>3/2 +\sqrt{\beta}/4 -4\gamma =3/2$ since $\gamma =\sqrt{\beta}/16.$
\end{proof}
Recall that by Lemma~\ref{lem1-ext} we may assume that for every $u\in U$ at most $2\sqrt{n}$ vertices $z\in N^-_D(u)\cap U$ satisfy $c(zu)\neq c_u$. Let $C_1=\{c_u : u \in U\}$ and let $C_2= c(E(G^*[U]))\setminus C_1.$

\begin{lemma}\label{lem5-ext}
Let $v, u, u'\in U$ be such that $u\neq u'$, $uv, u'v\in E(G^*[U])$ and $c(uv),c(u'v)\in C_2.$ Then for every $x,y \in V(G)$ there is a rainbow $x,y$-path in $G$ of length at most nine. 
\end{lemma}
\begin{proof} Since for every color $\alpha$, $F_\alpha$ is a star forest, $c_u\neq c_v$ or $c_{u'}\neq c_v$. Without loss of generality assume $c_u\neq c_v$. Let $\alpha :=c(uv).$ Since $\alpha\in C_2$, $ \alpha \neq c_u$ and by Lemma \ref{lemmaxu-ext} there is a rainbow $x,u$-path $P_1$ of length four that avoids colors $\alpha, c_v$. Then, by appealing to Lemma \ref{lemmaxu-ext} again, there is a rainbow $y,v$-path $P_2$ that avoids $V(P_1)$ and colors from $c(E(P_1))\cup \{\alpha\}.$ 
\end{proof}
Note that in Lemma \ref{lem5-ext} $c(uv)$ can be equal to $c(u'v).$
Recall that we are assuming that there exist two vertices in $G$ such that there is no rainbow path between them of length at most nine. In particular, by Lemma \ref{ext-case-W-W'} this implies that $|W|\geq 2.$
\begin{lemma}\label{lem6-ext} The following facts hold.
\begin{itemize}
\item[(a)] $W'=\emptyset$ and $W=\{v,w\}$ for some $v\neq w.$
\item[(b)] There are no arcs from $U$ to $W$ in $D'$. 
\item[(c)] If $zx$ is an arc in $D'$ with $z\in W$, then $x\in U$ and $c(zx)=c_x.$
\end{itemize}
\end{lemma}
\begin{proof} By Lemma \ref{lem5-ext}, 
$\left(\bigcup_{\alpha\in C_2} E(F_{\alpha})\right)\cap E(G^*[U])$
is a matching. By the definition of $D'$, for every $u\in U$ there is exactly one edge $uv\in E(G^*[U])$ such that $c(uv)=c_u$, and so the number of edges $e\in E(G^*[U])$ such that $c(e)\in C_1$ is equal to $|U|$. Consequently, $|E(G^*[U])|\leq 3|U|/2$ and by Lemma \ref{lem4-ext}, $W'=\emptyset$ and $|W|=2$ which proves (a).

For (b), suppose for some $u\in U$ and $z\in W$, $uz$ is an arc in $D'$. If $uz$ is not an arc in $D$, then there is an arc $u'u$ such that $u'\in U$ and $c(u'u)=c(uz)$. Then however $uu'$ is an option to be added to $D'$ instead of $uz$, contradicting the definition of $D'$. Consequently, $uz$ is in $D$. By Lemma \ref{lem2-ext} (b), for every $x\in V(G)$, there is a rainbow $x,z$-path in $G$ of length at most five. 
In addition, for every $u\in U$ and $z'\in W\setminus \{z\}$, by Lemma \ref{lem2-ext} (a), there is a rainbow $u,z'$-path in $G$ of length at most four. Therefore any two vertices in $G$ are connected by a rainbow path of length at most five. 

To prove (c), first note that if $\{x,z\}=W$ then there is a rainbow path between any two vertices in $G$ of length at most four by Lemma \ref{lem2-ext}. Thus $x\in U$. If $c(zx)\neq c_x$, then by Lemma \ref{lem2-ext} (a), for every $y\in V(G)$ there is a rainbow $y,x$-path of length at most four that avoids color $c(zx)$ and vertex $z$ which gives a rainbow $y,z$-path in $G$ of length at most five.
\end{proof}
\begin{lemma}\label{lem7-ext}
Let $u,u'\in U$ be such that $uu'\in E(G^*[U])$ and $c(uu')\in C_2$. Then $c_u=c_{u'}$.
\end{lemma}
\begin{proof}Suppose $c_u\neq c_{u'}$ and let
$\alpha:= c(uu')$. Then $c_u,c_{u'}, \alpha$ are distinct colors. Let $x, y\in V(G)$. By Lemma \ref{lemmaxu-ext}, there is a rainbow $x,u$-path $P_1$ of length at most fours that avoids colors $\alpha, c_{u'}$ and a rainbow $u',y$-path $P_2$ of length at most four that avoids $c(E(P_1))\cup \{\alpha\}$ and $V(P_1).$
\end{proof}
We are now ready to finish the proof of this extremal case. 
\begin{lemma}\label{ext-type2-lem}
    Let $\beta>0$ and $n\in \mathbb Z^+$ be such that $1/n \ll \beta \ll 1$. If $(G,c)$ is an edge-colored graph on $n$ vertices such that $\delta^c(G)\geq \frac{n}{2}$ which is $\beta$-extremal of type 2, then for every two distinct vertices  $x,y\in V(G)$ there is a rainbow $x,y$-path in $G$ of length at most nine.
\end{lemma}
\begin{proof} Suppose otherwise. Then $|W|\geq 2$. By Lemma \ref{lem6-ext}, $|W|=2$,  and the number of edges in $G^*[U]$ is at least 
$$\frac{n}{2}|U| - {|U|  \choose 2} = |U| \frac{n-|U|+1}{2}=\frac{3|U|}{2}.$$
As a result, there are $|U|$ edges in $G^*[U]$ of colors from $C_1$, 
$\left(\bigcup_{\alpha\in C_2} E(F_{\alpha})\right)\cap E(G^*[U])$
is a perfect matching, $|U|$ is even, and the underlying simple graph of $D'[U]$ is complete.

Moreover, if $c(uu')\in C_2$, then $c_u=c_{u'}$ by Lemma \ref{lem7-ext}. 
Consequently, for $v\in W$, by Lemma \ref{lem6-ext} (b) and (c), $d^c(v)\leq |U|/2= n/2-1$ contradicting (\ref{deltaD'}).
\end{proof}
\section{Proof of the Main Theorem}
We can now combine previous lemmas to prove the main result.
\begin{proof}[Proof of Theorem \ref{main}]
Let $\beta>0$ and $n\in \mathbb Z^+$ be such that Lemma \ref{ext-type1-lem}, Lemma \ref{ext-type2-lem}, and Lemma \ref{non-ext-lem} hold. Let $(G,c)$ be an edge color graph of order $n$ that satisfies $\delta^c(G)\geq n/2$. By deleting edges if necessary, we may assume that $(G,c)$ is edge-minimal. By Lemma \ref{non-ext-lem} if $(G,c)$ is not $\beta$-extremal of type 1 and not $\beta$-extremal of type 2 then for any two distinct vertices $x,y\in V(G)$ there is a rainbow $x,y$-path of length at most eight. If $(G,c)$ is $\beta$-extremal of type 1 or $\beta$-extremal of type 2, then by Lemma \ref{ext-type1-lem} and Lemma \ref{ext-type2-lem} for every $x,y\in V(G)$ with $x\neq y$ there is a rainbow $x,y$-path of length at most nine. 
\end{proof}

\section{Concluding remarks}
In Theorem~\ref{main} we assume that $n$ is sufficiently large. This assumption is not required in Theorem~\ref{thm-FM} but we need it for our stability argument. We believe that this assumption is not necessary. 
Another future direction is to improve the constant length in Theorem~\ref{main} and the coefficient in Corollary~\ref{cor-kconn}. We propose the following two questions:

\begin{question}
    Determiner the smallest integer $\ell$ such that for every sufficiently large $n$, if $(G, c)$ is an edge-colored graph of order $n$ and $\delta^c(G) \geq n/2$, then any two distinct vertices are connected by a rainbow path of length at most $\ell$.
\end{question}

Theorem~\ref{main} implies that $\ell$ is at most nine. However, we did not focus on optimizing this constant in the proof, and one might be able to improve it by looking into the extremal case of type 2.

\begin{question}
    Determine the smallest integer $f(k)$ such that the following holds. For every $k\in \mathbb{Z}^+$ there exists $n_0\in \mathbb{Z}^+$ such that for every $n\geq n_0$, if $(G,c)$ is an edge-colored graph on $n\geq n_0$ vertices and $\delta^c(G)\geq {n}/{2}+f(k)$, then $(G,c)$ is rainbow $k$-connected.  
\end{question}

Corollary~\ref{cor-kconn} implies that $f(k)\le 17(k-1)$, and the following example inspired by the extremal case of type 1 shows that $f(k)$ is a linear function in $k$: Suppose $n$ is even and take two disjoint copies of $K_{n/2}$ where each $K_{n/2}$ receives a rainbow edge-coloring. Add $k-1$ disjoint perfect matchings between the two copies of $K_{n/2}$, each perfect matchings receiving a unique new color. Then the minimum color degree is $n/2 + (k-2)$, and the graph is not rainbow $k$-connected.

\bibliographystyle{siam}
\bibliography{references}

\end{document}